\documentclass[12pt,a4wide]{amsart}
\usepackage{ifthen,amssymb,graphicx}
\usepackage[all]{xy}
\usepackage[usenames]{color}
\usepackage{mathrsfs}

\newcommand{\isom}{\stackrel{\sim}\to}

\newcommand{\catD}{{D}}
\newcommand{\dbc}[1]{{D}^b(#1)}
\newcommand{\dpc}[1]{{D}^+(#1)}
\newcommand{\dmc}[1]{{D}^-(#1)}
\newcommand{\cdbc}[1]{{D}^b_c(#1)}

\newcommand{\Hom}{{\operatorname{Hom}}}
\newcommand{\Aut}{{\operatorname{Aut}}}

\newcommand{\Eq}{{\operatorname{Eq}}}

\newcommand{\End}{{\operatorname{End}}}

\newcommand{\lotimes}{{\,\stackrel{\mathbf L}{\otimes}\,}}

\newcommand{\Id}{{\operatorname{Id}}}

\DeclareMathOperator{\Img}{{Im}}
\DeclareMathOperator{\Rea}{{Re}}

\DeclareMathOperator{\Ker}{{Ker}}
\DeclareMathOperator{\Or}{{O}}
\DeclareMathOperator{\Sp}{{Sp}}
\DeclareMathOperator{\U}{{U}}
\DeclareMathOperator{\Biext}{{Biext}}
\DeclareMathOperator{\Corr}{{Corr}}
\newcommand{\bbR}{{\mathbb R}}

\newcommand{\bbC}{{\mathbb C}}
\newcommand{\bbQ}{{\mathbb Q}}
\newcommand{\bbZ}{{\mathbb Z}}

\newcommand{\cO}{{\mathcal O}}

\newcommand{\bR}{{\mathbf R}}
\newcommand{\bL}{{\mathbf L}}

\newcommand{\cplx}[1]{{{\mathcal #1}^{\scriptscriptstyle\bullet}}}

\newtheorem{thm}{Theorem}[section]
\newtheorem*{thm*}{Theorem}
 \newtheorem*{prop*}{Proposition}
 \newtheorem*{claim*}{Claim}
\newtheorem{cor}[thm]{Corollary}
\newtheorem{lem}[thm]{Lemma}
\newtheorem{prop}[thm]{Proposition}

\theoremstyle{definition}
\newtheorem{defin}[thm]{Definition}
\theoremstyle{remark}
\newtheorem{rema}[thm]{Remark}
\newtheorem{exe}[thm]{Example}

\newenvironment{rem}{\begin{rema}}{\hfill\hspace{1pt}$\triangle$\end{rema}}

\numberwithin{equation}{section} 

\begin{document}
\title[Derived equivalences of Abelian varieties]{Derived equivalences of Abelian varieties and symplectic isomorphisms}

\date{\today}
\thanks {Work supported by research project MTM2013-45935-P (MINECO)}

\author[A. C. L\'opez Mart\'{\i}n]{Ana Cristina L\'opez Mart\'{\i}n}
\author[C. Tejero Prieto]{Carlos Tejero Prieto}
\email{anacris@usal.es, carlost@usal.es}
\address{Departamento de Matem\'aticas and Instituto Universitario de F\'{\i}sica Fundamental y Matem\'aticas
(IUFFyM), Universidad de Salamanca, Plaza de la Merced 1-4, 37008
Salamanca, Spain.}

\begin{abstract}  
We study derived equivalences of Abelian varieties in terms of their associated symplectic data. For simple Abelian varieties over an algebraically closed field of characteristic zero we prove that the natural correspondence introduced by Orlov, which maps equivalences to symplectic isomorphisms, is surjective.  \end{abstract}

\maketitle
\setcounter{tocdepth}{1}

{\small \tableofcontents }

\section*{Introduction}

Any Abelian variety $A$ has associated to it in a natural way a symplectic Abelian variety $A\times\widehat A$ whose symplectic structure is built out of the normalized Poincar\'e line bundle $\mathcal P$. Therefore we can associate to $A$ the group $\Sp(A\times \widehat A)\subset\Aut(A\times \widehat A)$ of symplectic automorphisms. More generally if we have another Abelian variety $B$ we have the space $\Sp(A\times \widehat A,B\times \widehat B)$ of symplectic isomorphisms. It is remarkable that these spaces of symplectic isomorphisms are deeply related with the spaces of derived equivalences between the Abelian varieties. The first connection was established by Polishchuk \cite{Pol96}, who in the case of an algebraically closed field proved that if there exists a symplectic isomorphism $f\in \Sp(A\times \widehat A,B\times \widehat B)$ then there exists an equivalence $\Phi\colon \cdbc{A}\to\cdbc{B}$; that is, $A$ and $B$ are Fourier-Mukai partners. After that, Orlov \cite{Or02} showed that there exists a natural correspondence that associates to any Fourier-Mukai functor $\Phi^{\cplx{K}}\colon \cdbc{A}\isom  \cdbc{B}$ a symplectic isomorphism $f_{\cplx{K}}\in \Sp(A\times \widehat A,B\times \widehat B)$. He claimed that this correspondence was surjective and with  this hypothesis he determined completely the group of derived equivalences of an Abelian variety as an extension of the symplectic group. More precisely, he showed that for any Abelian variety $A$ over an algebraically closed field of characteristic zero there is an exact sequence $$0\to \mathbb{Z}\oplus(A\times \hat{A})\to \Aut(\cdbc{A})\xrightarrow{\gamma_A} \Sp(A\times\hat{A})\to 1\,.$$  However, in order to prove that $\gamma_A$ is surjective, given a symplectic isomorphism $f\in \Sp(A\times\hat{A},B\times\widehat B)$ with matrix representation $f=\begin{pmatrix} \alpha &\beta\\
 \gamma & \delta
\end{pmatrix}$, one has to build out of it a kernel $\cplx{K}\in \cdbc{A\times B}$ such that $\Phi^{\cplx{K}}\colon \cdbc{A}\isom  \cdbc{B}$ is a derived equivalence and its associated symplectic automorphism $f_{\cplx{K}}$ is $f$. Orlov proves the existence of such a kernel under the assumption that $\beta\colon \widehat A\to B$ is an isogeny and  claims that if $f$ is such that $\beta$ is not an isogeny, then it is easy to see that $f$ can be written as $f=f_2\circ f_1$ where $f_1\in\Sp(A\times\widehat A)$, $f_2\in \Sp(A\times\widehat A,B\times\widehat B)$ and $\beta_1\colon \widehat A\to A$, $\beta_2\colon \widehat A\to B$ are isogenies. As far as we know there is no proof of this result in the literature. The main result of this paper, Theorem \ref{thm:orlov-corrected}, is a proof of the surjectivity of Orlov's correspondence for the building blocks of the category of Abelian varieties, namely simple Abelian varieties. We achieve this goal by studying the structure of symplectic isomorphisms between Abelian varieties over an arbitrary base field. 

The paper is organized as follows. In Section 1 we study symplectic structures on Abelian varieties. We start by motivating it with complex Abelian varieties and after that we do it in full generality following the ideas of Polishchuk with some simplifications due to a result of Grothendieck that identifies biextensions of Abelian schemes with divisorial correspondences. We also allow for fields that are not supposed to be algebraically closed. In Section 2 after recalling some basic facts about Fourier-Mukai functors we describe the kernels of these functors for Abelian varieties and show that they are semihomogeneous sheaves up to a shift. The final part of this section is devoted to the proof of the surjectivity of Orlov's correspondence for simple Abelian varieties.
\subsubsection*{Conventions}

For any scheme $X$ we denote by $\catD(X)$ the
derived category of complexes of $\cO_X$-modules with
quasi-coherent cohomology sheaves. Analogously $\dpc{X}$, $\dmc{X}$
and $\dbc{X}$ denote the derived categories of complexes
which are respectively bounded below, bounded above and bounded on
both sides, and have quasi-coherent cohomology sheaves. The
subscript $c$ refers to the corresponding subcategories of
complexes with coherent cohomology sheaves.

\section{The symplectic group of an Abelian variety}

The study of exact equivalences between derived categories of coherent sheaves $\cdbc{A}$ and $\cdbc{B}$ of two Abelian varieties $A$ and $B$ was carried out by Mukai \cite{Muk94}, Polishchuk \cite{Pol96} and Orlov \cite{Or02}. 

In \cite{Muk94}, Mukai studied the group of derived equivalences $\Aut(\cdbc{A})$ and he defined a discrete group $U(A\times \widehat A)$ that he called the unitary group of the Abelian variety $A$. Independently Polishchuk obtained in \cite{Pol96} the same group but deriving it from a symplectic point of view.

We start by explaining the introduction of this group in the particular case of complex Abelian varieties.

\subsection{A motivating example: the unitary group of a complex Abelian variety}\label{subsec:complex-Abelian-varieties}

Let $A$ be a complex Abelian variety. We denote an endomorphism ${f}\in \End(A\times\widehat A)$ as a matrix $${f}=\begin{pmatrix} \alpha & \beta\\ \gamma & \delta
\end{pmatrix},$$ where $\alpha\in\End(A)$, $\beta\in\Hom(\widehat A,A)$, $\gamma\in\Hom(A,\widehat A)$, $\delta\in\End(\widehat A)$. Each ${f}\in\End(A\times\widehat A)$ determines another endomorphism ${f}^\dagger\in\End(A\times\widehat A)$ whose matrix is $${f}^\dagger=\begin{pmatrix} \hat \delta & -\hat \beta\\ -\hat \gamma & \hat \alpha
\end{pmatrix},$$ where $\hat \alpha$, $\hat \beta$, $\hat \gamma$, $\hat \delta$ denote the transposed morphisms induced between the dual Abelian varieties and we are using the natural identification $k_A\colon A\xrightarrow{\sim}\widehat{\widehat A}$ induced by the normalized Poincar\'e line bundle $\mathcal P$ on $A\times \widehat A$. 

In connection with the study of the group of autoequivalences of the derived category of an Abelian variety, Mukai introduced in \cite{Muk94} the unitary group associated to an Abelian variety, whose definition we recall now.

\begin{defin} The unitary group associated to the Abelian variety $A$ is the group $U(A\times\widehat A)$ consisting of all automorphisms ${f}\in \Aut(A\times\widehat A)$ that satisfy $${f}^\dagger\circ{f}={f}\circ{f}^\dagger=\Id_{A\times\widehat A}.$$
\end{defin} 
\bigskip

\begin{rem} The unitary group has a natural involution $$\ddagger\colon U(A\times\widehat A)\to U(A\times\widehat A)$$ that maps $f=\begin{pmatrix} \alpha & \beta\\ \gamma & \delta
\end{pmatrix}$ to $f^\ddagger=\begin{pmatrix} \alpha & -\beta\\ -\gamma & \delta
\end{pmatrix}$. We will explain later the geometric origin of this involution.
\end{rem}

In his paper Mukai did not explain the reason for calling this a ``unitary'' group. We will show now that in fact $U(A\times\widehat A)$ is the unitary group of a certain Hermitian metric naturally associated to the Abelian variety. We will also see that it also can be understood in terms of symplectic geometry. In order to do this we need first to introduce some basic definitions that will also serve  to fix our notation.

A complex torus $A$ of dimension $g$ is defined by the following data: a free $\bbZ$-module $\Lambda$ of rank $2g$ and a complex structure $J\in\End_\bbR(V)$, $J^2=-\Id_V$, where $V=\bbR\otimes_\bbZ \Lambda$. Then $$A=(V/\Lambda,J)$$ is a $g$-dimensional complex torus. Notice that we have the canonical isomorphisms $$\Lambda=H_1(A,\bbZ),\quad V=H_1(A,\bbR).$$ We will add the subscript ``A'' to $\Lambda$, $V$, $J$ when we need to distinguish between different complex tori.

Given another complex torus $B=(V_B/\Lambda_B,J_B)$ a homomorphism ${f}\in\Hom(A,B)$ is defined by a $\bbZ$-module homomorphism $T\colon \Lambda_A\to\Lambda_B$ such that $$J_B\circ T_\bbR=T_\bbR\circ J_A\colon V_A\to V_B.$$ Therefore the Abelian group $\Hom(A,B)$ can be considered as a subgroup of $\Hom_\bbZ(\Lambda_A,\Lambda_B)$ giving rise to the rational representation $$\rho_r\colon \Hom(A,B)\hookrightarrow \Hom_\bbZ(\Lambda_A,\Lambda_B)$$ which maps ${f}\in\Hom(A,B)$ to $\rho_r({f})={f}_r=T\in\Hom_\bbZ(\Lambda_A,\Lambda_B)$. In the same way, we have the analytic representation $$\rho_a\colon \Hom(A,B)\hookrightarrow \Hom_\bbC((V_A,J_A),(V_B,J_B))$$ which maps ${f}\in\Hom(A,B)$ to $\rho_a({f})={f}_a=T_\bbR\in\Hom_\bbR(V_A,V_B)$. Notice that the rational representation is the restriction $({f})_a|_{\Lambda_A}=({f})_r$.

The dual torus $\widehat A$ is defined as follows. We set $\widehat V=\Hom_{\overline\bbC}((V,J),\bbC)$, the complex vector space of $\bbC$-antilinear forms endowed with its natural complex structure $\widehat J$ and we put $\widehat\Lambda:=\{\widehat\theta\in\widehat V\colon \Img\widehat\theta(\Lambda)\subseteq\bbZ\}$. Then we have $$\widehat A=(\widehat V/\widehat\Lambda,\widehat J).$$
Notice that the underlying real vector space of $\widehat V$ is naturally isomorphic to the real dual vector space $V^*=\Hom_\bbR(V,\bbR)$. The isomorphism is given by sending $\widehat\theta\in\widehat V$ to $\Img\widehat\theta\in V^*$. This gives an identification of $\widehat J$ with $-J^*$ and of  $\widehat\Lambda$ with $\Lambda^*=\Hom_\bbZ(\Lambda,\bbZ)$. Therefore $$\widehat A\simeq (V^*/\Lambda^*,-J^*).$$

Given a morphism of complex tori  ${f}\colon A\to B$, one easily checks that the transposed homomorphism $\widehat{f}\colon \widehat B\to\widehat A$ has an analytic representation $(\widehat{f})_a\colon \widehat V_B\to\widehat V_A$  given by $$(\widehat{f})_a=(({f})_a)^*.$$

In a similar way one shows that the natural isomorphism $k_A\colon A\xrightarrow{\sim}\widehat{\widehat A}$ has an analytic representation $(k_A)_a\colon V\xrightarrow{\sim}\widehat{\widehat V}$ that maps $v\in V$ to the complex antilinear form $\eta_v=(k_A)_a(v)\in \Hom_{\overline\bbC}(\widehat V,\bbC)$ defined by $\eta_v(\widehat\theta):=\overline{\widehat\theta(v)}$ for every $\widehat\theta\in \widehat V$.

By means of the natural isomorphisms $k_A\colon A\xrightarrow{\sim}\widehat{\widehat A}$, $k_B\colon B\xrightarrow{\sim}\widehat{\widehat B}$, one easily checks that the analytic representation of the doubly transposed morphism $\widehat{\widehat{f}}\colon\widehat{\widehat A}\to\widehat{\widehat B}$ gets identified with $({f})_a\colon V_A\to V_B$.

On $A\times \widehat A$ we have the normalized Poincar\'e bundle $\mathcal P$ which can be described in terms of its Appell--Humbert data $\mathcal P=L(H,\chi)$, where $H$ is the non-degenerate Hermitian form defined on the complex vector space $V\times \widehat V$  by $$H((v_1,\widehat\theta_1),(v_2,\widehat\theta_2))=\widehat\theta_1(v_2)+\overline{\widehat\theta_2(v_1)}$$ and $\chi\colon \Lambda\times\widehat\Lambda \to U(1)$ is the semicharacter  $\chi((\lambda,\widehat\lambda))=\exp(\pi i\Img(\widehat\lambda)(\lambda))$.

Let us recall that for any Hermitian form $H$ on a complex vector space $(V,J)$, its real part $g=\Rea H$ is a symmetric real bilinear form on $V$ and its imaginary part $\omega=\Img H$ is an alternating real $2$-form on $V$. Moreover, both of them are $J$-invariant and they determine each other by the following formulae: $$g(v_1,v_2)=\omega(J(v_1),v_2),\quad \omega(v_1,v_2)=g(v_1,J(v_2)),\quad \text{for any}\ v_1,v_2\in V.$$ This shows that $H$ is non-degenerate if and only if $g$ is non degenerate if and only if $\omega$ is a symplectic form.

The unitary group of a non degenerate Hermitian complex vector space $(V,H)$ is $$\U(V,H)=\{T\in\End_\bbC(V)\colon T^*H=H\}.$$ In the same way we have the real orthogonal group $$\Or(V,g)=\{T\in\End_\bbR(V)\colon T^*g=g\}$$ and the real symplectic group $$\Sp(V,\omega)=\{T\in\End_\bbR(V)\colon T^*\omega=\omega\}.$$ It is well known that $$\U(V,H)=\Or(V,g)\cap\Sp(V,\omega)=\Or(V,g)\cap\Aut_\bbC(V)=\Sp(V,\omega)\cap\Aut_\bbC(V).$$

The analytic representation allows us to identify $U(A\times\widehat A)$ with a subgroup of $\Aut_\bbC(V\times \widehat V)$. More precisely one has:

\begin{prop}Let $A$ be a complex Abelian variety and let $\mathcal P$ be the normalized Poincar\'e line bundle on $A\times\widehat A$ with Appell--Humbert data $\mathcal P=L(H,\chi)$. The analytic representation $\rho_a\colon \Aut(A\times\widehat A)\hookrightarrow \Aut_\bbC(V\times\widehat V)$ identifies the unitary group $U(A\times\widehat A)$ associated to the Abelian variety $A$ with the unitary group $U(V\times\widehat V,H)$.
\end{prop}

\begin{proof} If $H$ is the non-degenerate Hermitian form of the Apple-Humbert data of the normalized Poincar\'e bundle $\mathcal P=L(H,\chi)$, then its associated symplectic form $\omega$ is given by \begin{equation} \label{eq:symp-form}
\omega((v_1,\widehat\theta_1),(v_2,\widehat\theta_2))=\Img(\widehat\theta_1)(v_2)-\Img(\widehat\theta_2)(v_1).
\end{equation}
  Given ${f}\in U(A\times\widehat A)$ its analytic representation $\rho_a({f})\in \Aut_\bbC(V\times\widehat V)$ is complex linear. Therefore, it is enough to show that $$\rho_a(U(A\times\widehat A))\subset \Sp(V\times\widehat V,\omega).$$ Using the identifications previously explained in this section, especially the analytic representation of the natural isomorphism $k_A\colon A\xrightarrow{\sim}\widehat{\widehat A}$, a straightforward computation shows that $(\rho)_a(f)\in\Sp(V\times\widehat V,\omega)$ if and only if $$({f}^\dagger )_a\circ ({f})_a=({f})_a\circ({f}^\dagger)=\Id_{V\times\widehat V},$$ and so the claim follows.\end{proof}

Notice that the symplectic form $\omega$ introduced in \ref{eq:symp-form}, defines a $\bbZ$-valued alternating $2$-form on the lattice $\Lambda\times \widehat\Lambda$. Therefore one has the following:

\begin{cor} The rational representation  identifies the unitary group $U(A\times\widehat A)$ associated to the Abelian variety $A$ with the symplectic group $$\Sp(\Lambda\times\widehat \Lambda,\omega)\cap\rho_r(\Aut(A\times\widehat A)).$$
\end{cor}

Therefore, we see that the unitary group of a complex Abelian variety is also a symplectic group and thus it is reasonable to also denote it by $\Sp(A\times\widehat A)$ and call it the symplectic group of the Abelian variety $A$. Moreover, it is the symplectic point of view, rather than the unitary, that is the correct one to generalize to other fields.
We will explore this connection further in the next section when we explain the approach followed by Polishchuk for introducing the symplectic group associated to an Abelian variety over an arbitrary field. 
\begin{rem} Notice that the definition of the unitary or symplectic group can be given for any complex torus. The results obtained above continue to hold in this more general setting.
\end{rem}

The above constructions can be extended to a pair of complex Abelian varieties $A$, $B$. As before, any isomorphism $f\colon A\times \hat{A}\isom B\times\hat{B}$ can be written as a matrix  $\begin{pmatrix} \alpha &\beta\\
 \gamma & \delta
\end{pmatrix}$  where $\alpha\colon A\to B$, $\beta\colon \hat{A}\to B$, etc. are morphisms of Abelian varieties. One defines the isomorphism $$f^\dag\colon B\times \hat{B}\isom A\times\hat{A}\, ,$$ with matrix $\begin{pmatrix} \hat{\delta}&-\hat{\beta}\\
-\hat{\gamma} & \hat{\alpha}
\end{pmatrix}$.
We denote by $U(A\times\hat{A}, B\times\hat{B})$ the space formed by those isomorphisms $f\colon A\times\hat{A}\isom B\times \hat{B}$ that are isometric, that is, such that $f^\dag=f^{-1}$. The reason for this terminology is explained in the following Proposition, whose proof is similar to the case of a single Abelian variety. 

\begin{prop} Let us consider two complex Abelian varieties $A$, $B$ and let $\mathcal{P}_A$,  $\mathcal{P}_B$ be their normalized Poincar\'e line bundles with Appell--Humbert data $\mathcal P_A=L(H_A,\chi_A)$,  $\mathcal P_B=L(H_B,\chi_B)$. The analytic representations identify the space of isometric isomorphisms $U(A\times\widehat A, B\times\widehat B)$  with the space of unitary isomorphisms $$U((V_A\times\widehat V_A,H_A), (V_B\times\widehat V_B,H_B)).$$
\end{prop}

\begin{rem} As in the case of a single Abelian variety, there is an involution $$\ddagger\colon U(A\times\hat{A}, B\times\hat{B})\to U(A\times\hat{A}, B\times\hat{B})$$ that maps $f=\begin{pmatrix} \alpha & \beta\\ \gamma & \delta
\end{pmatrix}$ to $f^\ddagger=\begin{pmatrix} \alpha & -\beta\\ -\gamma & \delta
\end{pmatrix}$.
\end{rem}

\subsection{Abelian varieties over arbitrary fields: symplectic biextensions}

For general Abelian varieties Polishchuk, motivated by the classical theory of Heisenberg groups, introduced a symplectic formalism based on the notion of symplectic biextensions by the multiplicative group, see \cite{Pol96} and \cite[Chapter 10]{Pol03}. The notion of biextension was introduced by Mumford \cite{Mum69} in the context of formal groups and was later systematically studied by Grothendieck \cite{Gro72} in the framework of sheaves of groups over a site. We recall now the definitions of these concepts following \cite{Bre83}.

\begin{defin}Let $G_1$, $G_2$, $H$ be commutative group schemes over a base scheme $S$. A biextension of $G_1\times_S G_2$ by $H$ is an $H$-torsor $E$ over $G_1\times G_2$ together with sections of the induced $H$-torsors $$(m_{G_1}\times 1)^*E\wedge p_{13}^*E^{-1}\wedge p_{23}^*E^{-1},\ \text{and}\ (1\times m_{G_2})^*E\wedge p_{12}^*E^{-1}\wedge p_{13}^*E^{-1}$$ over $G_1\times G_1\times G_2$ and 
$G_1\times G_2\times G_2$, respectively.
\end{defin}

In addition these data must satisfy certain cocycle compatibility conditions, see \cite{Bre83} for the precise details. What is important for us is the following result of Grothendieck \cite[Corollaire 2.9.4, Example 2.9.5]{Gro72}:

\begin{prop}\label{prop:biextension-correspondences} Let $S$ be a scheme, $A$ and $B$ two Abelian schemes over $S$ and $\mathbb G_{m\, S}$ the multiplicative group scheme over $S$. Then there is a bijective correspondence between the space $\Biext(A\times_SB;\mathbb G_{m\,S})$ of biextensions of $A\times_S B$ by $\mathbb G_{m\, S}$ and the group $\Corr_S(A,B)$ of divisorial correspondences on $A\times_S B$.
\end{prop}

Let us recall that given two Abelian schemes $A$, $B$ over $S$, a divisorial correspondence from $A\to S$ to $B\to S$ is just a line bundle $L$ on $A\times_S B$ together with rigidifications $$\phi_A\colon s_A^*L\xrightarrow{\sim}\mathcal O_A,\quad \phi_B\colon s_B^*L\xrightarrow{\sim}\mathcal O_B$$ that coincide along the unit section $e_{A\times_S B}$ of $A\times_S B\to S$ $$e_A^*(\phi_A)=e_B^*(\phi_B)\colon e_{A\times_S B}^*L\xrightarrow{\sim}\mathcal O_S$$ where $s_A=(\Id_A, e_B)\colon A\to A\times_S B$, $s_B=(e_A, \Id_B)\colon B\to A\times_SB$ are the natural sections induced by the unit sections $e_A, e_B$ of $A$ and $B$, respectively.

It is well known, see \cite[Chapter 6, \S 2]{GIT},  that any divisorial correspondence $L\in\Corr_S(A,B)$ induces a morphism of Abelian schemes $\psi_L\colon A\to \widehat B$.

The approach followed by Polishchuk relied heavily on the machinery of biextensions. Rather than using it, we will use the equivalence result of Grothendieck to base our presentation in the much more concrete language of divisorial correspondences.

We start by expressing the notion of symplectic biextension introduced by Polishchuk in terms of correspondences.

\begin{defin} Let $X$ be an Abelian variety over a field $\Bbbk$. A skew-symmetric biextension of $X\times X$ is a correspondence $L\in \Corr_\Bbbk(X,X)$ together with an isomorphism of correspondences $$\varphi\colon \sigma^*L\xrightarrow{\sim}L^{-1}$$ where $\sigma\colon X\times X\to X\times X$ is the automorphism defined by permutation of factors, and a rigidification $\delta^*L\xrightarrow{\sim}\mathcal O_X$ over the diagonal immersion $\delta\colon X\to X\times X$ compatible with $\varphi$. The compatibility condition means that pulling-back the isomorphism $\varphi:\sigma^* L\xrightarrow{\sim} L^{-1}$ under  the diagonal immersion $\delta\colon X\to X\times X$, we obtain an isomorphism $\delta^*(\varphi) \colon \delta^*(L) \xrightarrow{\sim} \delta^*(L^{-1})$ that must be compatible with the given rigidification of $\rho^*(L)$, $\rho\colon \delta^*(L)\xrightarrow{\sim} \mathcal O_X$, and the natural rigidification of $\rho^*(L^{-1})$ induced by it, $\rho'\colon \delta^*(L^{-1})\xrightarrow{\sim} \mathcal O_X$. That is, one must have $$\rho=\rho'\circ \delta^*(\varphi).$$
A skew-symmetric biextension $L$ is called symplectic if the induced morphism $\psi_L\colon X\to\widehat X$ is an isomorphism and in this case we call the pair $(X,L)$ a symplectic Abelian variety.
\end{defin}

\begin{rem}\label{rem:isomorphi-correspondences} If $L\in\Corr_\Bbbk(X,X)$ is skew-symmetric, then $\widehat\psi_L=-\psi_L$. Moreover, given another $L'\in\Corr_\Bbbk(X,X)$ it is easy to see that one has  $\widehat\psi_L=\widehat\psi_{L'}$ if and only if $L$ and $L'$ are isomorphic line bundles.
\end{rem}

Once we have the notion of symplectic Abelian variety, we can mimic some of the definitions and constructions of symplectic geometry. 

\begin{defin} Let $(X,L)$ be a symplectic Abelian variety. We say that an Abelian subvariety $Y\subset X$ is isotropic if there exists an isomorphism of skew-symmetric biextensions  $L|_{Y\times Y}\simeq \mathcal O_{Y\times Y}$.
\end{defin}

\begin{rem} Notice that $Y\subset X$ is isotropic if and only if the composition $$Y\xrightarrow{i} X\xrightarrow{\psi_L} \widehat X\xrightarrow{\widehat i}\widehat Y$$ is zero.
\end{rem}

Let us recall that for any Abelian subvariety $i\colon A'\hookrightarrow A$ the theory of fppf quotients can be used to show that there exists a quotient $\pi\colon A\to A/A'$ such that $A/A'$ is an Abelian variety \cite[Lemma 1.7.4.4]{Lifting}. Moreover, proceeding as in \cite[Section 9.5]{Pol03} one shows that $\widehat i\colon \widehat A\to \widehat{A'}$ is surjective and $\widehat{A/A'}$ is contained in $\Ker\widehat i$. Therefore it makes sense to give the following:

\begin{defin} Let $(X,L)$ be a symplectic Abelian variety. We say that an Abelian subvariety $i\colon Y\hookrightarrow X$ is Lagrangian if it is isotropic and the morphism $Y\to \Ker\widehat i$ defined by $\psi_L$ induces an isomorphism $Y\xrightarrow{\sim} \widehat{X/Y}$. 
\end{defin}

If $A$ is an Abelian variety then $X=A\times \widehat A$ has a natural symplectic biextension that, as we will see in a moment, allows us to generalize to any Abelian variety the symplectic structure (\ref{eq:symp-form}) introduced for a complex Abelian variety. On $X\times X=A\times\widehat A \times A\times\widehat A$ we define $$L=p_{14}^*\mathcal P\otimes p_{32}^*\mathcal P^{-1}$$ where $\mathcal P$ is the normalized Poincar\'e line bundle on $A\times \widehat A$. It is well known that $\mathcal P$ is a divisorial correspondence on $A\times \widehat A$ and from this it follows easily that $L\in \Corr_\Bbbk(X\times\widehat X)$. Taking into account Proposition \ref{prop:biextension-correspondences} we get that  $L$ is a biextension of $X\times X$ and it is straightforward to check that it is skew-symmetric. Moreover, using the natural identification $k_A\colon A\xrightarrow{\sim} \widehat{\widehat A}$ induced by the Poincar\'e line bundle, one easily gets that the morphism $$\psi_L\colon A\times\widehat A\to \widehat A\times A$$ is given by $\psi_L(a,\hat a)=(-\hat a,a)$. Therefore we conclude that $L$ is a symplectic biextension.

\begin{defin} Let $A$ be an Abelian variety with Poincar\'e line bundle $\mathcal P$ over $X=A\times \widehat A$ and let $L=p_{14}^*\mathcal P\otimes p_{32}^*\mathcal P^{-1}$. We say that the pair $(X,L)$  is the natural symplectic data associated to $A$.
\end{defin}

\begin{rem} Let $A=(V/\Lambda,J)$ be a complex Abelian variety. Let us set $W=V\times\widehat V$, then $\widehat W=\widehat V\times V$ and we have a natural duality pairing $$\langle\,\ ,\ \rangle\colon W \times \widehat W\to\bbC$$  defined by $\langle(v_1,\widehat\theta_1),(\widehat\theta_2,v_2)\rangle=\widehat\theta_1(v_2)+\widehat\theta_2(v_1)$ for  $v_1,v_2\in V$, $\widehat\theta_1,\widehat\theta_2\in \widehat V$. Then the symplectic form $\omega$ defined in \ref{eq:symp-form} can be expressed as $$\omega=\Img(\langle\,\ ,\ \rangle\circ(\Id_W\times(\psi_L)_a)).$$ Since according to Remark \ref{rem:isomorphi-correspondences}, $\psi_L$ is completely determined by $L$, we can think of the symplectic biextension $L=p_{14}^*\mathcal P\otimes p_{32}^*\mathcal P^{-1}$ as an abstract incarnation of the symplectic form $\omega$ valid for any Abelian variety defined over an arbitrary field. \end{rem}

Now we are ready to generalize the symplectic group introduced in section \ref{subsec:complex-Abelian-varieties} for complex Abelian varieties to Abelian varieties over an arbitrary field.

\begin{defin} Let $A$ be an Abelian variety with symplectic data $(X,L)$. We say that an automorphism $f\in\Aut(A\times\widehat A)$ is symplectic if $(f\times f)^*L\simeq L$. We denote by $\Sp(A\times\widehat A)$ the group formed by the symplectic automorphisms.\end{defin}

As in Section \ref{subsec:complex-Abelian-varieties}, any endomorphism $f\in\End(A\times\widehat A)$ with matrix $f=\begin{pmatrix} \alpha & \beta\\ \gamma & \delta
\end{pmatrix},$  determines another endomorphism ${f}^\dagger\in\End(A\times\widehat A)$ whose matrix is ${f}^\dagger=\begin{pmatrix} \hat \delta & -\hat \beta\\ -\hat \gamma & \hat \alpha
\end{pmatrix}.$

\begin{prop} An automorphism $f\in\End(A\times\widehat A)$ is symplectic, $f\in \Sp(A\times\widehat A)$, if and only if $$f^\dagger\circ f=f\circ f^\dagger=\Id_{A\times \widehat A}.$$ \end{prop}

\begin{proof} According to Remark \ref{rem:isomorphi-correspondences}, $f$ is symplectic if and only if $$\psi_{(f\times f)^*L}=\psi_L.$$ A straightforward computation shows that $\psi_{(f\times f)^*L}=\widehat f\circ\psi_L\circ f$. Therefore, $f$ is symplectic if and only if $$\psi_L^{-1}\circ \widehat f\circ\psi_L\circ f=\Id_{A\times\widehat A}.$$ One easily shows that $f^\dagger=\psi_L^{-1}\circ \widehat f\circ\psi_L$ and so the proof is finished.
\end{proof}

\begin{defin} Let $A$ be an Abelian variety. The group $\Sp(A\times\widehat A)$ is called the symplectic group associated to the Abelian variety $A$.
\end{defin}

\begin{rem} Due to the way in which this group was originally introduced, it is usually denoted $U(A\times\widehat A)$ and one says that it is the unitary group of the Abelian variety $A$ although strictly speaking what generalizes to an arbitrary Abelian variety is the symplectic construction.
\end{rem}

The above ideas can be extended to a pair of Abelian varieties. 

\begin{defin} Let us consider two Abelian varieties $A$, $B$ with symplectic data $(X_A,L_A)$, $(X_B,L_B)$.  We say that an isomorphism of Abelian varieties $f\colon A\times \widehat A\to B\times\widehat B$ is symplectic if $(f\times f)^*L_B\simeq L_A$. We denote by $\Sp(A\times\widehat A,B\times \widehat B)$ the set formed by the symplectic isomorphisms. 
\end{defin}

As in the case of a single Abelian variety, if we have a morphism $f\colon A\times \hat{A}\isom B\times\hat{B}$, we define a new morphism $f^\dag\colon B\times \hat{B}\isom A\times\hat{A}$. Their matrices are related as before and in a similar way as above one shows that $f$ is symplectic if and only if it satisfies $$f^\dagger\circ f=\Id_{A\times \widehat A},\quad f\circ f^\dagger=\Id_{B\times\widehat B}.$$

We can explain now the geometric origin of the involution $$\ddagger\colon \Sp(A\times\widehat A,B\times \widehat B)\to \Sp(A\times\widehat A,B\times \widehat B).$$ Given $f\in \Sp(A\times\widehat A,B\times \widehat B)$, by swapping every Abelian variety with its dual we can think of it as an element $f^s\in \Sp(\widehat A\times A,\widehat B\times B)$. A straightforward computation proves the following result.

\begin{lem}\label{lem:ddagger} For every $f\in \Sp(A\times\widehat A,B\times \widehat B)$ one has a commutative diagram $$\xymatrix{A\times\widehat A\ar[r]^{\psi_{L_A}}_\sim\ar[d]_{f^\ddagger} &  \widehat A\times A\ar[d]^{f^s}\\ B\times\widehat B & \widehat B\times B\ar[l]_{\psi^{-1}_{L_B}}^\sim }$$ That is $f^\ddagger=\psi^{-1}_{L_B}\circ f^s\circ\psi_{L_A}$.
\end{lem}

\section{Derived equivalences of Abelian varieties}

Let  us consider a pair of Abelian varieties $A$ and $B$. Due to Orlov's Representability Theorem \cite{Or97}, for any exact equivalence of derived categories (we will simply say a derived equivalence) $$\Phi\colon  \cdbc{A}\isom \cdbc{B}$$ there exists a unique object (up to isomorphism) $\cplx{K}\in \cdbc{A \times B}$ such that $\Phi$ is isomorphic to the integral functor  $\Phi^\cplx{K}_{A\to B}\colon \cdbc{A}\isom \cdbc{B}$  defined by the kernel $\cplx{K}$ as $$\Phi_{A\to B}^{\cplx K}(\cplx{E})=\bR \pi_{2_\ast}(\bL \pi_1^\ast\cplx{E}\lotimes\cplx{K})\, ,$$ where $\pi_1\colon A\times B\to A$ and $\pi_2\colon A\times B\to B$ are the natural projections. Let us recall that an integral functor $\Phi^\cplx{K}_{A\to B}\colon \cdbc{A}\isom \cdbc{B}$ is called a Fourier-Mukai functor if it is a derived  equivalence. If in addition the kernel is a concentrated complex, that is a sheaf up to a shift, then the functor is called a Fourier-Mukai transform. We can rephrase Orlov's representability theorem by saying that all derived equivalences between Abelian varieties are Fourier-Mukai functors.
Finally, we recall that one says that two Abelian varieties $A$ and $B$ are Fourier-Mukai partners if there exists a Fourier-Mukai functor $\Phi^\cplx{K}_{A\to B}\colon \cdbc{A}\isom \cdbc{B}$ or equivalently if they are derived equivalent.

The first connection between the symplectic data of Abelian varieties over an algebraically closed field and their equivalences was obtained by Polishchuk \cite{Pol96}:

\begin{thm}\label{thm:polishchuk} Let $A$ and $B$ be Abelian varieties over an algebraically closed field. If there exists a symplectic isomorphism $f\colon A\times\widehat A\xrightarrow{\sim} B\times\widehat B$, that is $f\in\Sp(A\times\widehat A,B\times\widehat B)$, then $A$ and $B$ are Fourier-Mukai partners.
\end{thm}

The proof of this result given by Polishchuk is based on a theory of representations of the Heisenberg groupoid, which is a stack of monoidal groupoids naturally associated to the symplectic data of the Abelian varieties. The procedure mimics the classical theory of representations of the Heisenberg group. However the proof does not give precise information neither on the nature of the kernels of the Fourier-Mukai functors nor on the relationship between symplectic isomorphisms and the corresponding Fourier-Mukai functors.

On the other hand, Polishchuk also conjectured in \cite{Pol96} that for Abelian varieties over an algebraically closed field the converse statement of Theorem \ref{thm:polishchuk} is also true.

This question was studied immediately after by Orlov. He proved \cite[Theorem 2.10, Proposition 2.18]{Or02}:

\begin{thm}\label{prop:orlov} Let $A$, $B$ be Abelian varieties over an arbitrary field. There is a natural correspondence $$\gamma_{A,B}\colon\Eq(\cdbc{A},\cdbc{B})\to \Sp(A\times\widehat A,B\times\widehat B)$$ that associates to any Fourier-Mukai functor $\Phi^\cplx{K}_{A\to B}\colon \cdbc{A}\isom \cdbc{B}$ a symplectic isomorphism $f_\cplx{K}\in\Sp(A\times\widehat A,B\times\widehat B)$.\end{thm}

As a consequence he got \cite[Theorem 2.19]{Or02} the following stronger version of Polishchuk's conjecture:

\begin{thm}\label{thm:orlov} Let $A$, $B$ be Abelian varieties over an arbitrary field $\Bbbk$. If $A$ and $B$ are Fourier-Mukai partners then there exists a symplectic isomorphism $f\colon A\times\widehat A\xrightarrow{\sim} B\times\widehat B$.
\end{thm}

Thus, Theorem \ref{thm:polishchuk} together with Theorem \ref{thm:orlov} give the following:

\begin{thm}\label{thm:polishchuk-orlov} Two Abelian varieties $A$, $B$  over an algebraically closed field are Fourier-Mukai partners if and only if there exists a symplectic isomorphism $f\colon A\times\widehat A\xrightarrow{\sim} B\times\widehat B$, that is $f\in\Sp(A\times\widehat A,B\times\widehat B)$.
\end{thm}

Moreover, Orlov \cite[Corollary 2.16]{Or02} also proved:

\begin{prop} Let $A$ be an Abelian variety over an arbitrary field $\Bbbk$. The map 
$$\begin{aligned} \gamma_A\colon \Aut(\cdbc{A}) &\to \Sp(A\times \widehat{A})\\
\Phi^{\cplx{K}}&\mapsto f_{\cplx{K}} \, 
\end{aligned}$$ is a group morphism.\qed
\end{prop}

\subsection{Kernels of Fourier-Mukai functors and semihomogeneous sheaves}

As we mentioned above, Polishchuk's proof of Theorem \ref{thm:polishchuk} does not give information on the structure of the kernels that give Fourier-Mukai functors. The following result is due to Orlov \cite[Proposition 3.2]{Or02}.

\begin{prop} \label{VarAb} Let $A, B$ be Abelian varieties over an arbitrary field, and let $\cplx{K}$ be an object of
$\cdbc{A \times B}$ such that the integral functor $\Phi^\cplx{K}\colon \cdbc{A}\isom \cdbc{B}$ is a Fourier-Mukai functor.
Then $\cplx{K}$ has only one non-trivial cohomology sheaf, that is, $\cplx{K}\simeq \mathcal{K}[n]$  where $\mathcal{K}$ is a sheaf on $A \times B$   and $n\in \mathbb{Z}$. \end{prop} 

\begin{rem}As a consequence we get that all derived equivalences between Abelian varieties are Fourier-Mukai transforms.
\end{rem}

If $A$ is an Abelian variety, following  Mukai \cite{Muk78, Muk94}, for a coherent sheaf $\mathcal{E}$ on $A$ we consider the subgroup
$$\Upsilon(\mathcal{E})=\{ (a,\alpha)\in A\times \hat{A}\text{ such that } T_a^\ast \mathcal{E}\simeq \mathcal{E}\otimes \mathcal{P}_\alpha \} \subset A\times \hat{A}\, ,$$ where $T_a\colon A\to A$ denotes the translation by $a\in A$.
The sheaf $\mathcal{F}$ is said to be {\it semihomogeneous} if $\dim \Upsilon(\mathcal{E})=\dim A$. 

In a previous paper \cite[Proposition 3.2]{LST13} we have made more precise the nature of the kernels that induce derived equivalences.

\begin{prop} \label{flatness}Let $A, B$ be Abelian varieties over an arbitrary field. The sheaf $\mathcal{K}$ associated to an equivalence $$\Phi^{\cplx{K}}\colon \cdbc{A}\isom  \cdbc{B}$$ is a semihomogeneous sheaf and it
is flat over both factors.
\end{prop}

Therefore, every derived equivalence between Abelian varieties over an arbitrary field is induced by a Fourier-Mukai transform whose kernel is, up to a shift, a semihomogeneous sheaf. Taking into account the compatibility of Fourier-Mukai functors with moduli problems, it follows that under suitable circumstances semihomogeneous sheaves should correspond to universal families.

The following result of Orlov \cite[Theorem 4.14]{Or02} gives a complete description of the group of autoequivalences of the derived category $\cdbc{A}$ of an  Abelian variety $A$ over an algebraically closed field of characteristic zero.
\begin{thm}\label{thm:sequenceav} Let $A$ be an Abelian variety over an algebraically closed field $\Bbbk$ of characteristic zero. Then one has the following exact sequence of groups $$0\to \mathbb{Z}\oplus(A\times \hat{A})\to \Aut(\cdbc{A})\xrightarrow{\gamma_A} \Sp(A\times\hat{A})\to 1\, ,$$ where the autoequivalence defined by
 $(n, a, L)\in \mathbb{Z}\oplus (A\times \hat{A})$ is $$\Phi^{(n, a, L)}(\cplx{E})=T_{a\ast}(\cplx{E})\otimes L[n]\, .$$
\end{thm}

However the proof of this theorem given by Orlov is incomplete. Let us explain why it is so. The problem arises with proving the surjectivity of $\gamma_A$ and it depends on the more general:

\begin{claim*} For any pair of Abelian varieties $A$, $B$ over an algebraically closed field the map $$\gamma_{A,B}\colon\Eq(\cdbc{A},\cdbc{B})\to \Sp(A\times\widehat A,B\times\widehat B)$$ is surjective.
\end{claim*}
 
In order to prove it, given a symplectic isomorphism $f\in \Sp(A\times\hat{A},B\times\widehat B)$ with matrix representation $f=\begin{pmatrix} \alpha &\beta\\
 \gamma & \delta
\end{pmatrix}$, one has to build out of it a concentrated kernel $\cplx{K}\in \cdbc{A\times B}$ such that $\Phi^{\cplx{K}}\colon \cdbc{A}\isom  \cdbc{B}$ is a derived equivalence and its associated symplectic automorphism $f_{\cplx{K}}$ is $f$. Orlov proves the existence of such a kernel under the assumption that $\beta\colon \widehat A\to B$ is an isogeny \cite[Construction 4.10, Propositions 4.11, 4.12]{Or02}. After doing this, he claims \cite[pag. 591]{Or02} that if $f$ is such that $\beta$ is not an isogeny, then it is easy to see that $f$ can be written as $f=f_2\circ f_1$ where $f_1\in\Sp(A\times\widehat A)$, $f_2\in \Sp(A\times\widehat A,B\times\widehat B)$ and $\beta_1\colon \widehat A\to A$, $\beta_2\colon \widehat A\to B$ are isogenies. After quite some thought on this claim the authors convinced themselves that this statement is by no means ``easy to see''. Therefore further ideas are needed for its proof.

One could argue that the Claim must follow from Polishchuk's Theorem \ref{thm:polishchuk}. However this is not the case since as we have already mentioned, given $f\in \Sp(A\times\widehat A,B\times\widehat B)$ one does not know if the equivalence $\Phi^\cplx{K}\colon \cdbc{A}\to\cdbc{B}$, arising from Theorem \ref{thm:polishchuk}, verifies $f_\cplx{K}=f$. That is, we do not know if Polishchuk's procedure does give a section of the map  $\gamma_{A,B}\colon\Eq(\cdbc{A},\cdbc{B})\to \Sp(A\times\widehat A,B\times\widehat B)$.

In order to partially remedy the situation in what follows we give a direct proof of the Claim for the building blocks  of the category of Abelian varieties, namely simple Abelian varieties. We achieve this goal by studying the structure of symplectic isomorphisms between Abelian varieties over an arbitrary base field. For this we need first the following result that is interesting on its own.

\begin{prop}\label{prop:symplectic-isomorphism-isogenous} Let $A$, $B$ be Abelian varieties over an arbitrary field such that there exists a symplectic isomorphism $f\in\Sp(A\times\widehat A,B\times\widehat B)$, then $A$ and $B$ are isogenous $A\sim B$.
\end{prop}

\begin{proof} It is well known that every Abelian variety $A$ is isogenous to its dual Abelian variety $\widehat A$. Therefore by the complete reducibility theorem of Poincar\'e we have $A\sim A_1^{m_1}\times \cdots\times A_r^{m_r}$ for simple non isogenous Abelian varieties $A_1,\ldots, A_r$ unique up to isogenies and permutations. Since $A\sim\widehat A$ we get $A\times\widehat A\sim A_1^{2m_1}\times \cdots\times A_r^{2m_r}$. Repeating the argument with $B$ we obtain $B\times\widehat B\simeq B_1^{2n_1}\times \cdots\times B_s^{2n_s}$. Since $A\times \widehat A\simeq  B\times\widehat B$, once again by the complete reducibility theorem we obtain that $r=s$ and after a permutation we may assume $A_1\sim B_1, \ldots ,A_r\sim B_r$ and $m_1=n_1,\ldots m_r=n_r$. This implies that $A\sim B$ and the proof is finished.
\end{proof}

Given  Abelian varieties $A$, $B$ it is well known that $\Hom(A,B)$ is a torsion free $\bbZ$-module under addition of morphisms. Therefore we have a natural injection $$\Hom(A,B)\hookrightarrow\Hom^0(A,B):=\bbQ\otimes_\bbZ\Hom(A,B).$$  An element $\xi\in\Hom^0(A,B)$ is called a quasi-isogeny if there exists $\zeta\in \Hom^0(B,A)$ such that $\zeta\circ\xi=\Id_A$, $\xi\circ\zeta=\Id_B$. In this case we say that $A$ and $B$ are quasi-isogenous.

\begin{lem}\label{lem:prop-isogenous-ab-var} Let $A$, $B$ be Abelian varieties over an arbitrary field. One has the following properties:
\begin{enumerate}
\item $A$ and  $B$ are isogenous if and only if they are quasi-isogenous.
\item If $A$ and  $B$ are isogenous then $\End^0(A)\simeq \End^0(B)$.
\item An endomorphism $\varphi\in \End(X)$ is an isogeny if and only if it is invertible in the ring $\End^0(A)$.
\item $\End^0(A)$ is a division algebra if and only if every endomorphism $\alpha\in\End(A)$ is zero or an isogeny.
\end{enumerate}
\end{lem}

\begin{proof} $(1)$ If $\varphi\colon A\to B$ is an isogeny of degree $n$, it is well known that there exists an isogeny $\psi\colon B\to A$ such that $\psi\circ\varphi=n_A$, $\varphi\circ\psi=n_B$, and therefore $\varphi$ is a quasi-isogeny with inverse $\frac{\psi}{n}\in \Hom^0(B,A)$. The converse direction follows just by clearing denominators.

$(2)$ If $\xi\in \Hom^0(A,B)$ is a quasi isogeny with inverse $\zeta\in \Hom^0(B,A)$ then we define a ring morphism $$c\colon \End^0(A)\to\End^0(B)$$ as $c(f)=\xi\circ f\circ \zeta$ for any $f\in\End^0(A)$. This is a ring isomorphism with inverse $c^{-1}\colon\End^0(B)\to\End^0(A)$ defined as $c^{-1}(g)=\zeta\circ g\circ\xi$ for any $g\in\End^0(B)$.

$(3)$ follows from $(1)$ and $(4)$ is immediate from $(3)$.
\end{proof}

Let us recall that a non zero Abelian variety $A$ is simple if it has no proper Abelian subvarieties.  It is well known that if $A$ is simple then $\End^0(A)$ is a division algebra. More generally, if  $\End^0(A)$ is a division algebra then it follows from the complete reducibility theorem of Poincar\'e that $A$ is a simple Abelian variety. 

\begin{lem}\label{lem:isogenous-divison-algebra} Let $A$, $B$ be isogenous Abelian varieties over an arbitrary field. If $\End^0(A)$ is a division algebra then $\End^0(B)$ is also a division algebra and any morphism $\alpha\colon A\to B$ different from zero is an isogeny.
\end{lem}

\begin{proof} The first statement follows from part $(2)$ of Lemma \ref{lem:prop-isogenous-ab-var}. Let us consider now the second one. By hypothesis there exists an isogeny $\psi\colon B\to A$. Given a morphism $\alpha\colon A\to B$ then $\alpha \circ \psi \in \End^0(B)$ is either 0 or an isogeny due to $(4)$ of Lemma \ref{lem:prop-isogenous-ab-var}. If $\alpha\circ\psi=0$ then $\alpha=0$ as $\psi$ is surjective. On the other hand if $\alpha\circ\psi\colon B\to B$ is an isogeny then it follows that $\alpha$ is surjective since $\alpha\circ\psi$ is surjective. We conclude that $\alpha$ is an isogeny because $\dim(A)=\dim(B)$ since $A$ and $B$ are isogenous.
\end{proof}

Now we can give the key result: 

\begin{prop}\label{prop:key-result} Let $A$, $B$ Abelian varieties over an arbitrary field such that $\End^0(A)$ is a division algebra. If $f\in \Sp(A\times\hat{A},B\times\widehat B)$ is  a symplectic isomorphism  with matrix representation $f=\begin{pmatrix} \alpha &\beta\\
 \gamma & \delta
\end{pmatrix}$, then either $\beta\colon \widehat A\to B$ is an isogeny or $\alpha\colon A\xrightarrow{\sim}  B$ and $\delta\colon \widehat A\xrightarrow{\sim}\widehat B$ are isomorphisms and $\delta=\widehat\alpha^{-1}$.
\end{prop}

\begin{proof} By Proposition \ref{prop:symplectic-isomorphism-isogenous} we know that $A$ and $B$ are isogenous and thus $\widehat A$ and $B$ are also isogenous. Since by hypothesis $\End^0(A)$ is a division algebra, thanks to Lemma \ref{lem:isogenous-divison-algebra} it follows that if $\beta$ is not an isogeny then it must be zero. If $\beta=0$ then one has $$f=\begin{pmatrix} \alpha &0\\
 \gamma & \delta
\end{pmatrix},\quad f^\dagger=\begin{pmatrix} \widehat\delta &0\\
 -\widehat\gamma & \widehat\alpha
\end{pmatrix},$$ and the equalities $f^\dagger\circ f=\Id_{A\times\widehat A}$, $f\circ f^\dagger=\Id_{B\times\widehat B}$ give $\widehat\delta\circ \alpha=\Id_A$, $\alpha\circ\widehat\delta=\Id_B$. Therefore $\alpha$ and $\delta$ are isomorphisms such that $\delta=\widehat\alpha^{-1}$ and the proof is finished.
\end{proof}

As a consequence we can prove the Claim in the particular case announced above:

\begin{thm}\label{thm:orlov-corrected} If $A$, $B$ are Abelian varieties over an algebraically closed field of characteristic zero, such that either $\End^0(A)$ or $\End^0(B)$ is a division algebra, then the map $$\gamma_{A,B}\colon\Eq(\cdbc{A},\cdbc{B})\to \Sp(A\times\widehat A,B\times\widehat B)$$ is surjective.
\end{thm}

\begin{proof} Given $f=\begin{pmatrix} \alpha &\beta\\
 \gamma & \delta
\end{pmatrix}\in \Sp(A\times\widehat A,B\times\widehat B)$ by Proposition \ref{prop:key-result} either $\beta$ is an isogeny or $0$. In the first case we can apply Orlov's \cite[Construction 4.10]{Or02}. If $\beta=0$ and $\gamma\neq 0$  then $\gamma$ is an isogeny. By considering $f^\ddagger$ and  reversing the roles of $A$ and $\widehat A$ and of $B$ and $\widehat B$ we have $(f^\ddagger)^s\in \Sp(\widehat A\times A,\widehat B\times B)$. Now we can apply Orlov's construction to $(f^\ddagger)^s$ and conclude the existence of a kernel $\cplx{{\widehat K}}\in\cdbc{\widehat A\times\widehat B}$ inducing an equivalence $\Phi_{\widehat A\to \widehat B}^\cplx{{\widehat K}}\colon \cdbc{\widehat A}\to\cdbc{\widehat B}$ such that $f_\cplx{K}=(f^\ddagger)^s$. Now, we can use the Poincar\'e line bundles $\mathcal P_A$, and $\mathcal P_B$ as kernels to give equivalences $\Phi_{A\to \widehat A}^{\mathcal P_A}$ and  $\Phi_{B\to \widehat B}^{\mathcal P_B}$. Now we consider the equivalence $$\Phi^\cplx{K}:=(\Phi_{B\to \widehat B}^{\mathcal P_B})^{-1}\circ \Phi_{\widehat A\to \widehat B}^\cplx{{\widehat K}}\circ \Phi_{A\to \widehat A}^{\mathcal P_A}\colon \cdbc{A}\to\cdbc{B}$$ and use the well known fact $f_{\mathcal P_A}=\psi_{L_A}$, $f_{\mathcal P_B}=\psi_{L_B}$ to conclude that $$f_\cplx{K}=\psi_{L_B}^{-1}\circ (f^\ddagger)^s\circ\psi_{L_A}=(f^\ddagger)^\ddagger=f$$ where we have used Lemma \ref{lem:ddagger}.   Therefore we are left with the case $\beta=0$, $\gamma=0$. Now by Proposition \ref{prop:key-result}, $\alpha\colon A\xrightarrow{\sim} B$ is an isomorphism such that $\delta=\widehat\alpha^{-1}$, thus $f= \begin{pmatrix} \alpha &0\\
0 & \widehat\alpha^{-1}
\end{pmatrix}$ and this is well known to admit the kernel $\cplx{K}=(\alpha\times\Id_B)^*\mathcal O_{\Delta_B}$, where $\Delta_B$ is the diagonal $\Delta_B\hookrightarrow B\times B$. A straightforward computation shows that $f_\cplx{K}=f$ and so the proof is finished.
\end{proof}

We strongly believe that the key to prove the Claim for arbitrary Abelian varieties over an algebraically closed field of characteristic zero lies in pursuing further the symplectic approach.  We plan to address the general case in a future work.

\def\cprime{$'$}


\begin{thebibliography}{10}

  
\bibitem{Bre83}
{\sc L.~Breen}, {\em Fonctions th\^eta et th\'eor\`eme du cube},
Lecture Notes in Mathematics, 980. Springer-Verlag, Berlin, 1983.

\bibitem{Lifting} {\sc C.-L.~Chai,  B.~Conrad and F.~Oort},
Complex multiplication and lifting problems.
Mathematical Surveys and Monographs, 195. American Mathematical Society, Providence, RI, 2014.
  
\bibitem{Gro72}
{\sc A.~Grothendieck}, {\em Biextension de faisceaux de groupes}, in Groupes de
monodromie en g\'eom\'etrie alg\'ebrique. {I} \newblock S\'eminaire de G\'eom\'etrie Alg\'ebrique du Bois-Marie 1967Ð1969 (SGA 7 I). Dirig\'e par A. Grothendieck. Avec la collaboration de M. Raynaud et D. S. Rim. Lecture Notes in Mathematics, Vol. 288. Springer-Verlag, Berlin-New York, 1972, pp.~133--217.

\bibitem{LOZ96}
{\sc H.~W. Lenstra, Jr., F.~Oort, and Y.~G. Zarhin}, {\em Abelian
  subvarieties}, J. Algebra, 180 (1996), pp.~513--516.
  
\bibitem{LST13}
{\sc A.~C. L\'opez Mart\'{\i}n, D.~S\'anchez G\'omez, and C.~Tejero Prieto}, {\em Relative FourierÐMukai transforms for Weierstra\ss{} fibrations, Abelian schemes and Fano fibrations}, Math. Proc. Camb. Phil. Soc.  155 (2013), pp. 129--153.

\bibitem{Muk78}
{\sc S.~Mukai}, {\em Semi-homogeneous vector bundles on an {A}belian variety},
  J. Math. Kyoto Univ., 18 (1978), pp.~239--272.
  
\bibitem{Muk94}
{\sc S.~Mukai}, {\em Abelian variety and spin representation (in Japanese)}, Proceedings of the symposium ``Hodge theory and Algebraic Geometry'', Sapporo, (1994) pp.~110--135. English translation, University of Warwick, preprint 13, (1998).

\bibitem{Mum69}
{\sc D.~Mumford}, {\em Bi-extensions of formal groups}, in Algebraic Geometry (Internat. Colloq., Tata Inst. Fund. Res., Bombay, 1968), Oxford Univ. Press, London, 1969, pp. 307Ð322.

\bibitem{GIT}
{\sc D.~Mumford, J.~Fogarty, and F.~Kirwan}, {\em Geometric invariant theory},
  vol.~34 of Ergebnisse der Mathematik und ihrer Grenzgebiete (2), Springer-Verlag, Berlin, third~ed., 1994.

\bibitem{Or97}
{\sc D.~O. Orlov}, {\em Equivalences of derived categories and ${K}3$
  surfaces}, J. Math. Sci. (New York), 84 (1997), pp.~1361--1381.
\newblock Algebraic geometry, 7.

\bibitem{Or02}
\leavevmode\vrule height 2pt depth -1.6pt width 23pt, {\em Derived categories
  of coherent sheaves on Abelian varieties and equivalences between them}, Izv.
  Ross. Akad. Nauk Ser. Mat., 66 (2002), pp.~131--158.

\bibitem{Pol96}
{\sc A.~Polishchuk}, {\em Symplectic biextensions and a generalization of the Fourier-Mukai transform}, Math. Res. Lett. 3  no. 6, (1996), pp.~813--828. 

\bibitem{Pol03}
{\sc A.~Polishchuk}, {\em Abelian varieties, theta functions and the Fourier transform}, Cambridge University Press, Cambridge, 2003.

\end{thebibliography}
\end{document}